\documentclass[letterpaper, 10 pt, conference]{ieeeconf}  

\IEEEoverridecommandlockouts                              
\usepackage{amsmath,amssymb,enumerate,subfigure,multirow,hhline,color}
\usepackage{xcolor}
\usepackage{hyperref}
\usepackage{graphicx}
\usepackage{graphics}
\usepackage[sort]{cite}

\usepackage{algorithm}
\usepackage{algpseudocode}

\usepackage{epsfig,psfrag}
\usepackage{tabularx}
\usepackage{tabulary}

\usepackage[sort]{cite}

\usepackage{hyperref}
\hypersetup{breaklinks=true,colorlinks=true,linkcolor=black,citecolor=black}

\usepackage[noabbrev,capitalise,nameinlink]{cleveref}

\DeclareMathOperator{\argmax}{arg\,max}

\newtheorem{theorem}{Theorem}

\newtheorem{assumption}{Assumption}

\newtheorem{lemma}{Lemma}

\newtheorem{proposition}{Proposition}

\allowdisplaybreaks

\title{\LARGE \bf
On Some Geometric Behavior of Value Iteration on the Orthant: Switching System Perspective
}

\author{Donghwan Lee
\thanks{D. Lee is with the Department of Electrical and Engineering, Korea Advanced Institute of Science and Technology (KAIST), Daejeon, 34141, South Korea {\tt\small
donghwan@kaist.ac.kr}.}
\thanks{This material was supported in part by the National Research Foundation under Grant NRF-2021R1F1A1061613 and in part by the Institute of Information communications Technology Planning Evaluation (IITP) grant funded by the Korea government
(MSIT)(No.2022-0-00469)}
}

\begin{document}

\maketitle
\thispagestyle{empty}
\pagestyle{empty}

\begin{abstract}
In this paper, the primary goal is to offer additional insights into the value iteration through the lens of switching system models in the control community. These models establish a connection between value iteration and switching system theory and reveal additional geometric behaviors of value iteration in solving discounted Markov decision problems. Specifically, the main contributions of this paper are twofold: 1) We provide a switching system model of value iteration and, based on it, offer a different proof for the contraction property of the value iteration. 2) Furthermore, from the additional insights, new geometric behaviors of value iteration are proven when the initial iterate lies in a special region. We anticipate that the proposed perspectives might have the potential to be a useful tool, applicable in various settings. Therefore, further development of these methods could be a valuable avenue for future research.

\end{abstract}

\section{Introduction}
Dynamic programming~\cite{bertsekas1996neuro,bertsekas2015dynamic} is a general and effective methodology for finding an optimal solution for Markov decision problems~\cite{puterman2014markov}. Value iteration is a popular class of dynamic programming algorithms, and its exponential convergence is a fundamental and well-established~\cite{bertsekas1996neuro,bertsekas2015dynamic} result through the contraction mapping property of the Bellman operator.

In this paper, we offer some additional insights on the value iteration based on switching system models~\cite{liberzon2003switching,lee2017periodic,lee2018stabilizability,Geromel2006a,Hu2008,lin2009stability,Zhang2009a,lee2020unified} in the control community. These models provide new geometric behaviors of value iteration to solve Markov decision problems~\cite{puterman2014markov}. In particular, the main contributions of this paper are twofold: 1) We present a switching system model of value iteration and, based on it, provide a different proof for the contraction property of value iteration. 2) Furthermore, from the new perspectives and frameworks, we study additional geometric behaviors of value iteration when the initial iterate lies in a special region. More specifically, the detailed contributions are summarized as follows:
\begin{enumerate}
\item A switching system model of value iteration is introduced along with its special properties. Based on it, we offer a different proof for the following contraction property of the value iteration:
\begin{align}
\left\| {Q_{k+1}  - Q^* } \right\|_\infty   \le \gamma \left\| {Q_k  - Q^* } \right\|_\infty,\label{eq:0}
\end{align}
where $Q_k$ is the $k$th iterate of the value iteration, $Q^*$ is the optimal Q-function, and $\gamma\in (0,1)$ is the so-called discount factor in the underlying Markov decision problem. Although this result is fundamental and classical, the proof technique utilized in this paper is distinct from the classical approaches, offering additional insights.

\item Furthermore, from the additional insights provided in the first part, we prove new geometric behaviors of value iteration when the initial iterate lies in a special region. Specifically, when the initial iterate $Q_0$ is within the set $Q_0 \leq Q^*$ (shifted orthant), the iterates of the value iteration remain in the set, displaying different behaviors than the cases with a general initial iterate. When $Q_0 \leq Q^* $, one can establish the new contraction property
\begin{align}
\left\| {Q_{k+1}  - Q^* } \right\|_M  \le (\gamma  + \varepsilon ) \left\| {Q_k - Q^* } \right\|_M\label{eq:2}
\end{align}
where $\left\| ( \cdot )\right\|_M = \sqrt {( \cdot )^T M( \cdot )}$ is the weighted Euclidean norm, $\varepsilon >0$ is any real number such that $\gamma + \varepsilon \in (0,1)$, and $M$ is a positive definite matrix dependent on $\varepsilon >0$. This result corresponds to a contraction mapping property of the Bellman operator in terms of the weighted Euclidean norm, which cannot be derived from the classical approaches. Geometrically, the sublevel sets corresponding to the infinity norm $\left\| \cdot \right\|_\infty$ are squares, while the sublevel sets corresponding to the weighted Euclidean norm $\left\| \cdot \right\|_M$ are ellipsoids.

A more notable result is that when $Q_0 \leq Q^*$, the iterates satisfy the linear inequality
\begin{align}
(\gamma  + \varepsilon )v^T (Q_k  - Q^* ) \le v^T (Q_{k + 1}  - Q^* ) \le 0,\quad \forall k \ge 0,\label{eq:7}
\end{align}
where $\varepsilon >0$ is any real number such that $\gamma + \varepsilon \in (0,1)$, and $v$ is a positive vector dependent on $\varepsilon$. Geometrically, this implies that the component along the direction of $v$ in $Q_k - Q^*$ diminishes exponentially.
\end{enumerate}

The analysis presented in this paper employs insights and tools from switching systems~\cite{liberzon2003switching,lee2017periodic,lee2018stabilizability,Geromel2006a,Hu2008,lin2009stability,Zhang2009a,lee2020unified}, positive switching systems~\cite{blanchini2012co,fornasini2010linear,zhang2013stabilization,fornasini2011stability}, nonlinear systems~\cite{khalil2002nonlinear}, and linear systems~\cite{chen1995linear}. We emphasize that our goal in this paper is to offer additional insights and analysis templates for value iteration via its connections to switching systems, rather than improving existing convergence rates. We anticipate that the proposed switching system perspectives may have the potential to be a useful tool, applicable in various settings, and stimulate the synergy between dynamic system theory and dynamic programming. Thus, further development of these methods could be a valuable avenue for future research. Finally, some ideas in this paper have been inspired by~\cite{lee2020unified,lee2021discrete}, which develop switching system models of Q-learning~\cite{sutton1998reinforcement}. However, most technical results and derivation processes are nontrivial and entirely different from those in~\cite{lee2020unified,lee2021discrete}. Another related work is~\cite{guo2022convex}, which provides a linear programming-based analysis of value iteration and also employs positive switching system models for their analysis. However, the previous work merely proposed a numerical linear programming-based analysis to verify convergence properties and did not obtain the explicit relations given in~\eqref{eq:0}, \eqref{eq:2} ,\eqref{eq:7}. Therefore, our analysis offers significantly different results than those in~\cite{guo2022convex}.

\section{Preliminaries}\label{sec:preliminaries}

\subsection{Notations}
The adopted notation is as follows: ${\mathbb R}$: set of real numbers; ${\mathbb R}^n $: $n$-dimensional Euclidean
space; ${\mathbb R}^{n \times m}$: set of all $n \times m$ real
matrices; $A^T$: transpose of matrix $A$; $A \succ 0$ ($A \prec
0$, $A\succeq 0$, and $A\preceq 0$, respectively): symmetric
positive definite (negative definite, positive semi-definite, and
negative semi-definite, respectively) matrix $A$; $I$: identity matrix with appropriate dimensions; $|{\cal S}|$: cardinality of a finite set $\cal S$; $A \otimes B$: Kronecker’s product of matrices $A$ and $B$; $\bf 1$: the vector with ones in all elements; $\lambda _{\max } ( \cdot )$: maximum eigenvalue; $\lambda _{\min} ( \cdot )$: minimum eigenvalue.

\subsection{Markov decision problem}
We consider the infinite-horizon discounted Markov decision problem (MDP)~\cite{puterman2014markov}, where the agent sequentially takes actions to maximize cumulative discounted rewards. In a Markov decision process with the state-space ${\cal S}:={ 1,2,\ldots ,|{\cal S}|}$ and action-space ${\cal A}:= {1,2,\ldots,|{\cal A}|}$, the decision maker selects an action $a \in {\cal A}$ with the current state $s$. The state then transitions to a state $s'$ with probability $P(s'|s,a)$, and the transition incurs a reward $r(s,a,s')$, where $r$ is a reward function. For convenience, we consider a deterministic reward function and simply write $r(s_k,a_k ,s_{k + 1}) =:r_k, k\geq 0$. A deterministic policy, $\pi :{\cal S} \to {\cal A}$, maps a state $s \in {\cal S}$ to an action $\pi(s)\in {\cal A}$. The objective of the Markov decision problem (MDP) is to find an optimal (deterministic) policy, $\pi^*$, such that the cumulative discounted rewards over infinite time horizons are maximized, i.e.,
\begin{align*}
\pi^*:= \argmax_{\pi\in \Theta} {\mathbb E}\left[\left.\sum_{k=0}^\infty {\gamma^k r_k}\right|\pi\right],
\end{align*}
where $\gamma \in [0,1)$ is the discount factor, $\Theta$ is the set of all admissible deterministic policies, $(s_0,a_0,s_1,a_1,\ldots)$ is a state-action trajectory generated by the Markov chain under policy $\pi$, and ${\mathbb E}[\cdot|\pi]$ represents an expectation conditioned on the policy $\pi$. The Q-function under policy $\pi$ is defined as
\begin{align*}
&Q^{\pi}(s,a)={\mathbb E}\left[ \left. \sum_{k=0}^\infty {\gamma^k r_k} \right|s_0=s,a_0=a,\pi \right]
\end{align*}
for all $s\in {\cal S},a\in {\cal A}$, and the optimal Q-function is defined as $Q^*(s,a)=Q^{\pi^*}(s,a)$ for all $s\in {\cal S},a\in {\cal A}$. Once $Q^*$ is known, then an optimal policy can be retrieved by $\pi^*(s)=\argmax_{a\in {\cal A}}Q^*(s,a)$.
Throughout, we make the following standard assumption.
\begin{assumption}\label{assumption:bounded-reward} The reward function is unit bounded as follows:
\begin{align*}
\max _{(s,a,s') \in {\cal S} \times {\cal A} \times {\cal S}} |r (s,a,s')| \leq 1.
\end{align*}
\end{assumption}
The unit bound imposed on $r$ is just for simplicity of analysis.
Under~\cref{assumption:bounded-reward}, Q-function is bounded as follows.
\begin{lemma}\label{thm:3}
We have $\|Q^*\|_\infty \leq \frac{1}{1-\gamma}$.
\end{lemma}
\begin{proof}
The proof is straightforward from the definition of Q-function as follows: $|Q(s,a)| \le \sum\limits_{k = 0}^\infty  {\gamma ^k }  = \frac{1}{{1 - \gamma }}$.
\end{proof}

\subsection{Switching system}
Let us consider the \emph{switched linear system (SLS)}~\cite{liberzon2003switching,lee2017periodic,lee2018stabilizability,Geromel2006a,Hu2008,lin2009stability,Zhang2009a,lee2020unified},
\begin{align}
&x_{k+1}=A_{\sigma_k} x_k,\quad x_0=z\in {\mathbb
R}^n,\quad k\in \{0,1,\ldots \},\label{eq:switched-system}
\end{align}
Where $x_k \in \mathbb{R}^n$ is the state, $\sigma \in \mathcal{M} := {1, 2, \ldots, M}$ is called the mode, $\sigma_k \in \mathcal{M}$ is called the switching signal, and ${A_\sigma, \sigma \in \mathcal{M}}$ are called the subsystem matrices. The switching signal can be either arbitrary or controlled by the user under a certain switching policy. Specifically, a state-feedback switching policy is denoted by $\sigma_k = \sigma(x_k)$. The analysis and control synthesis of SLSs have been actively studied during the last decades~\cite{liberzon2003switching, lee2017periodic, lee2018stabilizability, Geromel2006a, Hu2008, lin2009stability, Zhang2009a, lee2020unified}.
A more general class of systems is the \emph{switched affine system (SAS)}
\begin{align*}
&x_{k+1}=A_{\sigma_k} x_k + b_{\sigma_k},\quad x_0=z\in {\mathbb
R}^n,\quad k\in \{0,1,\ldots \},
\end{align*}
where $b_{\sigma_k} \in {\mathbb R}^n$ is the additional input vector, which also switches according to $\sigma_k$. Due to the additional input $b_{\sigma_k}$, its stabilization becomes much more challenging.
Lastly, when all elements of the subsystem matrices, $\{A_\sigma,\sigma\in {\mathcal M}\}$, are nonnegative, then SLS is called positive SLS~\cite{blanchini2012co,fornasini2010linear,zhang2013stabilization,fornasini2011stability}.
\subsection{Definitions}\label{sec:1}
Throughout the paper, we will use the following compact notations:
\begin{align*}
P:=& \begin{bmatrix}
   P_1\\
   \vdots\\
   P_{|{\cal A}|}\\
\end{bmatrix},\; R:= \begin{bmatrix}
   R(\cdot,1) \\
   \vdots \\
   R(\cdot,|A|) \\
\end{bmatrix},
\; Q:= \begin{bmatrix}
   Q(\cdot,1)\\
  \vdots\\
   Q(\cdot,|{\cal A}|)\\
\end{bmatrix},
\end{align*}
where $P_a = P(\cdot|\cdot,a)\in {\mathbb R}^{|{\cal S}| \times |{\cal S}|}$, $Q(\cdot,a)\in {\mathbb R}^{|{\cal S}|},a\in {\cal A}$, and $R \in {\mathbb R}^{|{\cal S}||{\cal A}|}$ is an enumerate of $R(s,a): = {\mathbb E}[r_k |s_k  = s,a_k  = a]$ with an appropriate order compatible with other definitions. Note that $P\in{\mathbb R}^{|{\cal S}||{\cal A}| \times |{\cal S}|}$, and $Q\in {\mathbb R}^{|{\cal S}||{\cal A}|}$. In this notation, Q-function is encoded as a single vector $Q \in {\mathbb R}^{|{\cal S}||{\cal A}|}$, which enumerates $Q(s,a)$ for all $s \in S$ and $a \in A$ with an appropriate order. In particular, the single value $Q(s,a)$ can be written as
$
Q(s,a) = (e_a  \otimes e_s )^T Q,
$
where $e_s \in {\mathbb R}^{|{\cal S}|}$ and $e_a \in {\mathbb R}^{|{\cal A}|}$ are $s$-th basis vector (all components are $0$ except for the $s$-th component which is $1$) and $a$-th basis vector, respectively. Therefore, in the above definitions, all entries are ordered compatible with this vector $Q$.

For any stochastic policy, $\pi:{\cal S}\to \Delta_{|{\cal S}|}$, where $\Delta_{|{\cal A}|}$ is the set of all probability distributions over ${\cal A}$, we define the corresponding action transition matrix as
\begin{align}
\Pi^\pi:=\begin{bmatrix}
   \pi(1)^T \otimes e_1^T\\
   \pi(2)^T \otimes e_2^T\\
    \vdots\\
   \pi(|S|)^T \otimes e_{|{\cal S}|}^T \\
\end{bmatrix}\in {\mathbb R}^{|{\cal S}| \times |{\cal S}||{\cal A}|},\label{eq:swtiching-matrix}
\end{align}
where $e_s \in \mathbb{R}^{|{\cal S}|}$.
Then, it is well-known that
$
P\Pi^\pi \in \mathbb{R}^{|{\cal S}||{\cal A}| \times |{\cal S}||{\cal A}|}
$
is the transition probability matrix of the state-action pair under policy $\pi$.
If we consider a deterministic policy, $\pi:{\cal S}\to {\cal A}$, the stochastic policy can be replaced with the corresponding one-hot encoding vector
$
\vec{\pi}(s):=e_{\pi(s)}\in \Delta_{|{\cal A}|},
$
where $e_a \in \mathbb{R}^{|{\cal A}|}$, and the corresponding action transition matrix is identical to~\eqref{eq:swtiching-matrix} with $\pi$ replaced with $\vec{\pi}$. For any given $Q \in \mathbb{R}^{|{\cal S}||{\cal A}|}$, denote the greedy policy with respect to $Q$ as $\pi_Q(s):=\argmax_{a\in {\cal A}} Q(s,a)\in {\cal A}$.
We will use the following shorthand frequently:
$\Pi_Q:=\Pi^{\pi_Q}.$

\subsection{Q-value iteration (Q-VI)}
In this paper, we consider the so-called Q-value iteration (Q-VI)~\cite{bertsekas1996neuro} given in~\cref{algo:standard-Q-learning2}, where $F$ is called Bellman operator.
It is well-known that the iterates of Q-VI converge exponentially to $Q^*$ in terms of the infinity norm $\left\| \cdot \right\|_\infty$~\cite[Lemma~2.5]{bertsekas1996neuro}.
\begin{lemma}\label{thm:1}
We have the bounds for Q-VI iterates $\left\| {Q_{k+1}  - Q^* } \right\|_\infty   \le \gamma \left\| {Q_k  - Q^* } \right\|_\infty$.
\end{lemma}
The proof is given in~\cite[Lemma~2.5]{bertsekas1996neuro}, which is based on the contraction property of Bellman operator. Note that~\cite[Lemma~2.5]{bertsekas1996neuro} deals with the value iteration for value function instead of Q-function addressed in our work. However, it is equivalent to Q-VI. Next, a direct consequence of~\cref{thm:1} is the convergence of Q-VI
\begin{align}
\left\| {Q_k  - Q^* } \right\|_\infty   \le \gamma^k \left\| {Q_0  - Q^* } \right\|_\infty\label{eq:3}
\end{align}

\begin{algorithm}[t]
\caption{Q-VI}
  \begin{algorithmic}[1]
    \State Initialize $Q_0 \in {\mathbb R}^{|{\cal S}||{\cal A}|}$ randomly.
    \For{iteration $k=0,1,\ldots$}
        \State Update
        \[
Q_{k + 1} (s,a) = \underbrace {R(s,a) + \gamma \sum\limits_{s' \in S} {P(s'|s,a)\max _{a' \in A} Q_k (s',a')}}_{ = :F Q_k }
\]

    \EndFor
  \end{algorithmic}\label{algo:standard-Q-learning2}
\end{algorithm}

In what follows, an equivalent switching system model that captures the behavior of Q-VI is introduced, and based on it, we provide a different proof of \cref{thm:1}.

\section{Switching system model}\label{sec:convergence}

In this section, we study a discrete-time switching system model of Q-VI and establish its finite-time convergence based on the stability analysis of the switching system.
Using the notation introduced in \cref{sec:1}, the update in~\cref{algo:standard-Q-learning2} can be rewritten as
\begin{align}
Q_{k+1}= R+\gamma P\Pi_{Q_k}Q_k ,\label{eq:1}
\end{align}
Recall the definitions $\pi_Q(s)$ and $\Pi_Q$. Invoking the optimal Bellman equation $(\gamma P\Pi_{Q^*}-I)Q^*+R=0$,~\eqref{eq:1} can be further rewritten by
\begin{align}
(Q_{k + 1}  - Q^* ) = \gamma P\Pi _{Q_k } (Q_k  - Q^* ) + \gamma P(\Pi _{Q_k }  - \Pi _{Q^* } )Q^*,\label{eq:Q-learning-stochastic-recursion-form}
\end{align}
which is a SAS (switched affine system). In particular, for any $Q \in {\mathbb R}^{|{\cal S}||{\cal A}|}$, define
\begin{align*}
A_Q : = \gamma P\Pi _Q ,\quad b_Q : = \gamma P(\Pi _Q  - \Pi _{Q^* } )Q^*
\end{align*}
Hence, Q-VI can be concisely represented as the SAS
\begin{align}
Q_{k + 1}- Q^* = A_{Q_k} (Q_k - Q^*) + b_{Q_k},\label{eq:swithcing-system-form}
\end{align}
where $A_{Q_k}$ and $b_{Q_k}$ switch among matrices from $\{\gamma P\Pi^\pi:\pi\in \Theta\}$ and vectors from $\{\gamma P(\Pi^\pi - \Pi^{\pi^*})Q^* :\pi\in\Theta \}$ according to the changes of $Q_k$. Therefore, the convergence of Q-VI is now reduced to analyzing the stability of the above SAS. Th main obstacle in proving the stability arises from the presence of the affine term. Without it, we can easily establish the exponential stability of the corresponding deterministic switching system, under arbitrary switching policy. Specifically, we have the following result.
\begin{proposition}\label{prop:stability}
For arbitrary $H_k \in {\mathbb R}^{|{\cal S}||{\cal A}|}, k\ge 0$, the linear switching system $Q_{k+1} - Q^* = A_{H_k} (Q_k - Q^*), Q_0 - Q^*\in {\mathbb R}^{|{\cal S}||{\cal A}|}$, is exponentially stable such that $\|Q_{k+1}- Q^*\|_\infty\le \gamma \|Q_k - Q^*\|_\infty, k \ge 0$ and $\|Q_k- Q^*\|_\infty\le \gamma ^k \|Q_0 - Q^*\|_\infty, k \ge 0$.
\end{proposition}

The above result follows immediately from the key fact that $\|A_Q \|_\infty \le \gamma$, which we formally state in the lemma below.
\begin{lemma}\label{lemma:max-norm-system-matrix}
For any $Q \in {\mathbb R}^{|{\cal S}||{\cal A}|}$, $\|A_Q \|_\infty \le \gamma$, where the matrix norm  $\| A \|_\infty :=\max_{1\le i \le m} \sum_{j=1}^n {|A_{ij} |}$ and $A_{ij}$ is the element of $A$ in $i$-th row and $j$-th column.
\end{lemma}
\begin{proof}
Note $\sum_j|[A_Q]_{ij}|=\sum_j {| [\gamma P\Pi _Q ]_{ij}|}= \gamma$, which completes the proof.
\end{proof}

However, because of the additional affine term in the original switching system~\eqref{eq:swithcing-system-form}, it is not obvious how to directly derive its finite-time convergence. To circumvent the difficulty with the affine term, we will resort to two simpler upper and lower bounds, which are given below.
\begin{proposition}[Upper and lower bounds]\label{prop:lower-bound2}
For all $k\geq0$, we have
\[
A_{Q^* } (Q_k  - Q^* ) \le Q_{k + 1}  - Q^* \le A_{Q_k } (Q_k  - Q^* ).
\]
\end{proposition}
\begin{proof}
For the lower bound, we have
\begin{align*}
&(Q_{k+1}- Q^*)\\
=&A_{Q^*} (Q_k-Q^*)+(A_{Q_k}-A_{Q^*}) (Q_k-Q^*)+b_{Q_k}\\
 =&A_{Q^*} (Q_k-Q^*) + \gamma P(\Pi_{Q_k}-\Pi_{Q^*})Q_k\\
\ge & A_{Q^*} (Q_k-Q^*)
\end{align*}
where the third line is due to $P(\Pi_{Q_k}- \Pi_{Q^*})Q_k\ge P(\Pi_{Q^*}-\Pi_{Q^*})Q_k=0$.
For the upper bound, one gets
\begin{align*}
(Q_{k+1}- Q^*)
=& A_{Q_k}(Q_k- Q^*)+b_{Q_k}\\
\leq& A_{Q_k}(Q_k- Q^*),
\end{align*}
where we used the fact that $b_{Q_k}=(\gamma P\Pi_{Q_k} Q^*-\gamma P\Pi_{Q^*} Q^*)\le (\gamma P\Pi_{Q^*} Q^*-\gamma P\Pi_{Q^*} Q^*)=0$ in the first inequality. This completes the proof.
\end{proof}

The lower bound is in the form of an LTI system. Specifically, the lower bound corresponds to a special LTI system form known as a positive linear system~\cite{roszak2009necessary}, where all the elements of the system matrix $A_{Q^*}$ are nonnegative.
In our analysis, this property will be utilized to derive new behaviors of Q-VI.
Similarly, the system matrix $A_{Q_k}$ in the upper bound switches according to the changes in $Q_k$.
Therefore, it can be proven that the upper bound is in the form of positive SLSs~\cite{blanchini2012co, fornasini2010linear, zhang2013stabilization, fornasini2011stability}, where all the elements of the system matrix for each mode are nonnegative.

\subsection{Finite-time error bound of Q-VI}
In this subsection, we provide a proof for the contraction property outlined in~\cref{thm:1}. Although this result represents one of the most basic and fundamental outcomes in classical dynamic programming, we offer an alternative proof in this paper. This proof, grounded in the switching system model discussed in earlier sections, is provided below.

Since $A_{Q^* } (Q_k  - Q^* ) \le Q_{k + 1}  - Q^* \le A_{Q_k } (Q_k  - Q^* )$ from~\cref{prop:lower-bound2}, it follows that $(e_a  \otimes e_s )^T A_{Q^* } (Q_k  - Q^* ) \le (e_a  \otimes e_s )^T (Q_k  - Q^* ) \le (e_a  \otimes e_s )^T A_{Q_k } (Q_k  - Q^* )$.
If $(e_a  \otimes e_s )^T (Q_k  - Q^* ) \le 0$, then $|(e_a  \otimes e_s )^T (Q_k  - Q^* )| \le |(e_a  \otimes e_s )^T A_{Q^* } (Q_k  - Q^* )|$, where $e_s \in {\mathbb R}^{|{\cal S}|}$ and $e_a \in {\mathbb R}^{|{\cal A}|}$ are the $s$-th and $a$-th standard basis vectors, respectively. If $(e_a  \otimes e_s )^T (Q_k  - Q^* ) > 0$, then $|(e_a  \otimes e_s )^T (Q_k  - Q^* )| \le |(e_a  \otimes e_s )^T A_{Q_k } (Q_k  - Q^* )|$. Therefore, one gets
\begin{align*}
&\left\| {Q_{k + 1}  - Q^* } \right\|_\infty\\
\le& \max \{ \left\| {A_{Q_k } (Q_k  - Q^* )} \right\|_\infty  ,\left\| {A_{Q*} (Q_k  - Q^* )} \right\|_\infty  \}\\
\le& \max \{ \gamma \left\| {Q_k  - Q^* } \right\|_\infty  ,\gamma \left\| {Q_k  - Q^* } \right\|_\infty  \}\\
=& \gamma \left\| {Q_k  - Q^* } \right\|_\infty,
\end{align*}
which completes the proof.

\section{Convergence of Q-VI on the orthant}
In the previous section, we revisited the convergence of Q-VI from the perspective of switching system viewpoints. In this section, we study additional geometric behaviors of Q-VI, again using the switching system perspective. Let us suppose that $Q^* \ge Q_0$ holds. Please note that such an initial value can be found by setting
\[
Q_0  = -\frac{1}{{1 - \gamma }}{\bf{1}}
\]
because
\[
 - \frac{1}{{1 - \gamma }}{\bf{1}} \le Q^*  \le \frac{1}{{1 - \gamma }}{\bf{1}}
\]
holds according to ~\cref{thm:3}. Then, using the bound provided in ~\cref{prop:lower-bound2}, it can be proved that $Q^* \ge Q_k,\forall k \geq 0$. This is equivalent to saying $0 \ge Q_k - Q^*,\forall k \geq 0$. In other words, if the initial iterate $Q_0$ falls within the shifted orthant, $Q^* \ge Q_0$, then future iterates will also stay within the same set, i.e., $Q^* \ge Q_k,\forall k \geq 0$.
\begin{proposition}\label{thm:2}
Suppose that $Q_0 - Q^* \le 0$ holds. Then, $Q_k - Q^* \le 0, \forall k \ge 0$.
\end{proposition}
\begin{proof}
For an induction argument, suppose $Q_k - Q^* \le 0$ for any $k \ge 0$. Then, by~\cref{prop:lower-bound2}, it follows that $Q_{k + 1}  - Q^*  \le A_{Q_k } (Q_k  - Q^* ) \le 0$ because $A_{Q_k }$ is a nonnegative matrix. Therefore, $Q_{k+1}  - Q^* \le 0$, and the proof is completed by induction.
\end{proof}

From the results above, it can be observed that the behavior of $Q_k - Q^*$ is fully dictated by the lower bound, which is a linear function of $Q_k - Q^*$. Therefore, the fundamental theory of linear systems can be applied to analyze the behavior of Q-VI.
To facilitate this, the subsequent result reviews the Lyapunov theory for discrete-time linear systems.
\begin{proposition}\label{prop:Lyapunov-theorem}
For any $\varepsilon >0$ such that $\gamma + \varepsilon \in (0,1)$, there exists the corresponding positive definite $M\succ 0 $ such that
\[
A_{Q^*}^T M A_{Q^*} =  (\gamma + \varepsilon)^2 (M - I),
\]
and
\begin{align*}
\lambda_{\min}(M) \ge 1,\; \lambda_{\max}(M) \le \frac{|{\cal S}||{\cal A}|}{1-\left({\frac{\gamma}{\gamma+\varepsilon} }\right)^2}.
\end{align*}
\end{proposition}
The proof of ~\cref{prop:Lyapunov-theorem} is provided in the Appendix. Note that ~\cref{prop:Lyapunov-theorem} can be applied to general LTI systems. However, because the lower bound takes the form of positive linear systems, it can be demonstrated that the Lyapunov matrix $M$ also possesses the additional special property outlined below.
\begin{proposition}
$M$ is a nonnegative matrix.
\end{proposition}
\begin{proof}
The proof is easily done from the construction in~\eqref{eq:4}.
\end{proof}

From the perspectives provided above, we can derive a bound on $Q_k - Q^*$ in terms of the weighted Euclidean norm $\left\| \cdot \right\|_M$. This is an alternative to the infinity norm $\left\| \cdot \right\|_\infty$, which cannot be derived from classical contraction mapping arguments.
\begin{theorem}\label{quadratic-bound}
Suppose that $Q_0 - Q^* \le 0$ holds. Then, for any $k \geq 0$,
\[
\left\| {Q_{k+1}  - Q^* } \right\|_M  \le (\gamma  + \varepsilon ) \left\| {Q_k  - Q^* } \right\|_M 
\]
holds
\end{theorem}
\begin{proof}
We have
\begin{align*}
&(Q_{k + 1}  - Q^* )^T M(Q_{k + 1}  - Q^* )\\
\le& (Q_k  - Q^* )^T A_{Q^* }^T MA_{Q^* } (Q_k - Q^* )\\
= & (\gamma  + \varepsilon )^2 (Q_k  - Q^* )^T M(Q_k  - Q^* )\\
& - (\gamma  + \varepsilon )^2 (Q_k  - Q^* )^T (Q_k  - Q^* )\\
\le & (\gamma  + \varepsilon )^2 (Q_k  - Q^* )^T M(Q_k  - Q^* )
\end{align*}
where the first inequality follows from the fact that $M$ is a nonnegative matrix and $A_{Q^* } (Q_k  - Q^* ) \le Q_{k + 1}  - Q^*\leq 0$ from~\cref{prop:lower-bound2}, and the equality is due to the Lyapunov theorem in~\cref{prop:Lyapunov-theorem}.
Taking the squared root on both sides yields the desired conclusion.
\end{proof}

The result in~\cref{quadratic-bound} provides
\[
\left\| {Q_{k+1}  - Q^* } \right\|_M  \le (\gamma  + \varepsilon ) \left\| {Q_k - Q^* } \right\|_M 
\]
for an arbitrarily small $\varepsilon>0$, which implies
\[
\left\| {Q_{k}  - Q^* } \right\|_M  \le (\gamma  + \varepsilon )^k \left\| {Q_0 - Q^* } \right\|_M.
\]
This relationship reveals different geometric convergence behaviors: the sublevel sets associated with the infinity norm $\left\| \cdot \right\|_\infty$ are squares, whereas those corresponding to the weighted Euclidean norm $\left\| \cdot \right\|_M$ are ellipsoids.

Furthermore, our new approach enables us to derive a bound on a linear function of $Q_k - Q^*$. Specifically, for positive LTI systems, one common Lyapunov function is a linear Lyapunov function of the form $V(x) = v^T x$ for a positive vector $v \in {\mathbb R}^{|{\cal S}|\times |{\cal A}|}, v \geq 0$. We proceed by providing an explicit form of such a vector $v$.
\begin{proposition}\label{prop:Lyapunov-theorem2}
For any $\varepsilon >0$ such that $\gamma + \varepsilon \in (0,1)$ and for any positive vector $w \in {\mathbb R}^{|{\cal S}\times {\cal A}|}$, define
\begin{align*}
v: = \left( {\sum\limits_{i = 0}^\infty  {\left( {\frac{1}{{\gamma  + \varepsilon }}} \right)^i A_{Q^* }^i } } \right)^T w\in {\mathbb R}^{|{\cal S}\times {\cal A}|}
\end{align*}

Then, it holds true that
\[
v > 0,\quad \left\| w \right\|_\infty \le \left\| {v  } \right\|_\infty   \le \frac{\left\| w \right\|_1}{{1 - \gamma }}
\]
and
\begin{align}
v ^T A_{Q^*}   = (\gamma  + \varepsilon )(v^T  - w^T ).\label{eq:5}
\end{align}
\end{proposition}
The proof is given in Appendix. \cref{prop:Lyapunov-theorem2} proves that $V(x) = v^T x $ plays the role of a linear Lyapunov function for the positive LTI system. From~\cref{prop:Lyapunov-theorem2}, one can prove another convergence result of Q-VI.
\begin{theorem}\label{thm:4}
Suppose that $Q_0 - Q^* \le 0$ holds. For the positive vector $v$ given in~\cref{prop:Lyapunov-theorem2}, we have
\begin{align*}
(\gamma  + \varepsilon )v^T (Q_k  - Q^* )\le v^T (Q_{k + 1}  - Q^* ) \le 0
\end{align*}
for all $k \ge 0$.
\end{theorem}
\begin{proof}
Multiplying both sides of~\eqref{eq:5} in~\cref{prop:Lyapunov-theorem2} by $Q_k - Q^*$ from the right leads to
\begin{align*}
&(\gamma  + \varepsilon )v^T (Q_k  - Q^* ) - (\gamma  + \varepsilon )w^T (Q_k  - Q^* )\\
=& v^T A_{Q^* } (Q_k  - Q^* )\\
\le& v^T (Q_{k + 1}  - Q^* )\\
 \le& 0
\end{align*}
where the first equality comes from~\eqref{eq:5}, the first inequality is due to $A_{Q^* } (Q_k  - Q^* ) \le Q_{k + 1}  - Q^*$ in~\cref{prop:lower-bound2}, $v\geq 0$, and the second inequality follows from $Q_{k + 1}  - Q^*  \le 0$ in~\cref{thm:2}. Noting that $- (\gamma  + \varepsilon )w^T (Q_k  - Q^* )\geq 0$, the result implies 
\begin{align*}
(\gamma  + \varepsilon )v^T (Q_k  - Q^* )\le v^T (Q_{k + 1}  - Q^* ) \le 0.
\end{align*}
This completes the proof.
\end{proof}

\cref{thm:4} implies
\begin{align}
(\gamma  + \varepsilon )^k v^T (Q_0  - Q^* ) \le v^T (Q_k  - Q^* ) \le 0,\quad \forall k \ge 0.\label{eq:6}
\end{align}

Suppose that $Q_k - Q^* = [v]_k + [v]_k^\bot$, where $[v]_k$ represents the component of $Q_k - Q^*$ in the direction of $v$, and $[v]_k^\bot$ represents the components of $Q_k - Q^*$ orthogonal to $v$. Then, according to~\eqref{eq:6}, the component of $Q_k - Q^*$ in the direction of $v$ diminishes exponentially. From the construction of $v$ in ~\cref{prop:Lyapunov-theorem2}, it can be seen that there could be infinitely many such vectors $v$, depending on the positive vector $w$. Therefore, $Q_k - Q^*$ becomes trapped in the intersections of the nonpositive orthant and infinitely many half planes. An illustration of a single half plane is provided in~\cref{fig:1}.
\begin{figure}[ht]
\centering\includegraphics[width=8cm,height=8cm]{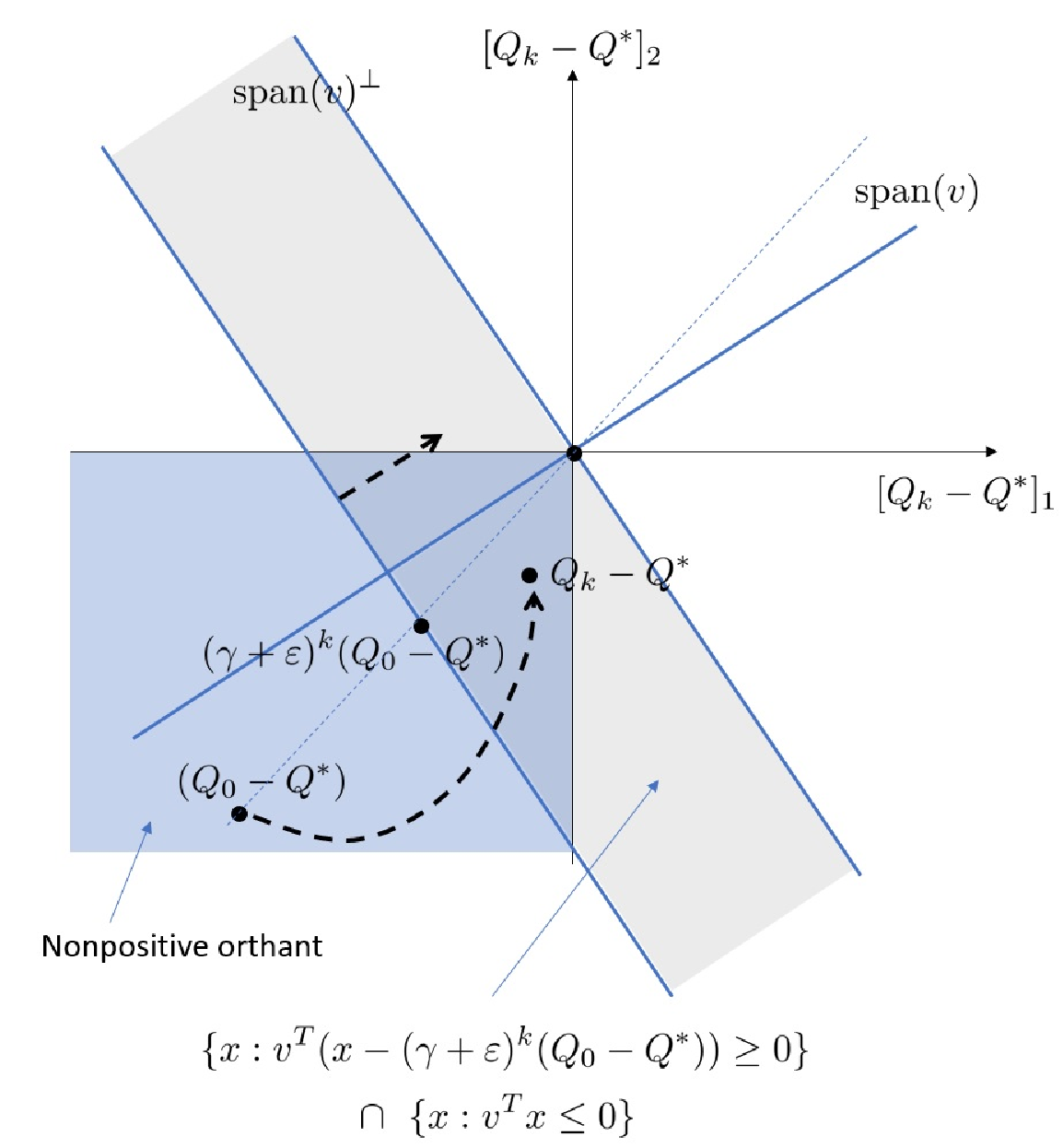}
\caption{Evolution of $Q_k-Q^*$ and geometric properties from~\cref{thm:4}. }\label{fig:1}
\end{figure}

\section{Conclusion}\label{sec:conclusion}
In this paper, we have presented additional insights on value iteration, approached through the lens of switching system models in the control community. This offers a connection between value iteration and switching system theory and reveals additional geometric behaviors of value iteration. Specifically, we introduced a switching system model of value iteration and, based on it, provided a novel proof for the contraction property of the value iteration. Moreover, our insights led to the proof of new geometric behaviors of value iteration when the initial iterate resides in a particular set (the shifted orthant). We believe that the perspectives proposed here could serve as useful tools applicable in various settings. As such, further development of these methods may present a valuable future direction.

\bibliographystyle{IEEEtran}
\bibliography{reference}

\section{Appendix}

\subsection{Proof of Proposition~\ref{prop:Lyapunov-theorem}}\label{app:prop-Lyapunov}

\begin{proof}
For simplicity, denote $A=A_{Q^*}$. Consider matrix $M$ such that
\begin{align}
M = \sum_{k=0}^\infty {\left({\frac{1}{\gamma+\varepsilon} }\right)^{2k} (A^k)^T A^k}. \label{eq:4}
\end{align}
Noting that
\begin{align*}
&(\gamma+\varepsilon)^{-2} A^T M A + I\\
=& \frac{1}{(\gamma+\varepsilon)^2}A^T \left(\sum_{k=0}^\infty {\left( {\frac{1}{\gamma+\varepsilon}} \right)^{2k} (A^k)^T A^k }\right)A + I\\
=& M,
\end{align*}
we have $(\gamma+\varepsilon)^{-2} A^T M A + I = M$, resulting in the desired conclusion. Next, it remains to prove the existence of $M$ by proving its boundedness. Taking the norm on $M$ leads to
\begin{align*}
\left\| P \right\|_2  =& \left\| {I + (\gamma  + \varepsilon )^{ - 2} A^T A + (\gamma  + \varepsilon )^{ - 4} (A^2 )^T A^2  +  \cdots } \right\|_2\\
\le& \left\| I \right\|_2  + (\gamma  + \varepsilon )^{ - 2} \left\| {A^T A} \right\|_2\\
&  + (\gamma  + \varepsilon )^{ - 4} \left\| {(A^2 )^T A^2 } \right\|_2  +  \cdots\\
=& \left\| I \right\|_2  + (\gamma  + \varepsilon )^{ - 2} \left\| A \right\|_2^2  + (\gamma  + \varepsilon )^{ - 4} \left\| {A^2 } \right\|_2^2  +  \cdots\\
=& 1 + |{\cal S}||{\cal A}|(\gamma  + \varepsilon )^{ - 2} \left\| A \right\|_\infty ^2\\
&  + |S||A|(\gamma  + \varepsilon )^{ - 4} \left\| {A^2 } \right\|_\infty ^2  +  \cdots\\
=& 1 - |{\cal S}||{\cal A}| + \frac{{|{\cal S}||{\cal A}|}}{{1 - \left( {\frac{\gamma }{{\gamma  + \varepsilon }}} \right)^2 }}.
\end{align*}

Finally, we prove the bounds on the maximum and minimum eigenvalues.
From the definition~\eqref{eq:4}, $M \succeq I$, and hence $\lambda_{\min}(M)\ge 1$. On the other hand, one gets
\begin{align*}
&\lambda_{\max}(M)\\
=& \lambda_{\max}(I + (\gamma+\varepsilon)^{-2} A^T A\\
&+ (\gamma+\varepsilon)^{-4}(A^2)^T A^2+\cdots)\\
\le& \lambda_{\max}(I) + (\gamma+\varepsilon)^{-2} \lambda_{\max}(A^T A)\\
&+ (\gamma+\varepsilon)^{-4}\lambda_{\max}((A^2)^T A^2 )+\cdots\\
=& \lambda_{\max}(I) + (\gamma+\varepsilon)^{-2} \|A\|_2^2 + (\gamma+\varepsilon)^{-4} \|A^2\|_2^2  +  \cdots\\
\le& 1 + |{\cal S}||{\cal A}|(\gamma+\varepsilon)^{-2} \|A\|_\infty^2 \\
& + |{\cal S}||{\cal A}|(\gamma+\varepsilon)^{-4} \|A^2\|_\infty ^2 + \cdots\\
\le& \frac{|{\cal S}||{\cal A}|}{1 - \left(\frac{\gamma}{\gamma+\varepsilon} \right)^2},
\end{align*}
where $|\mathcal{S}|$ and $ |\mathcal{A}|$ denote the cardinality of the sets $\cal S$ and $\cal A$, respectively.
The proof is completed.
\end{proof}

\subsection{Proof of Proposition~\ref{prop:Lyapunov-theorem2}}

\begin{proof}
We have{\small
\begin{align*}
v^T A_{Q^* }  =& w^T \left( {\sum\limits_{i = 0}^\infty  {\left( {\frac{1}{{\gamma  + \varepsilon }}} \right)^i A_{Q^* }^i } } \right)A_{Q^* }\\
 =& (\gamma  + \varepsilon )w^T \left( {\sum\limits_{i = 1}^\infty  {\left( {\frac{1}{{\gamma  + \varepsilon }}} \right)^i A_{Q^* }^i } } \right)\\
 =& (\gamma  + \varepsilon )\left\{ {w^T \left( {\sum\limits_{i = 0}^\infty  {\left( {\frac{1}{{\gamma  + \varepsilon }}} \right)^i A_{Q^* }^i } } \right) - w^T } \right\}\\
 =& (\gamma  + \varepsilon )(v^T  - w^T )
\end{align*}}

Moreover,{\small
\begin{align*}
\left\| v \right\|_\infty   =& \left\| {w^T \left( {\sum\limits_{i = 0}^\infty  {A_{Q^* }^i } } \right)} \right\|_\infty\\
 \le& \left\| w^T \right\|_\infty  \left\| {\sum\limits_{i = 0}^\infty  {A_{Q^* }^i } } \right\|_\infty\\
 \le& \| w \|_1 \sum\limits_{i = 0}^\infty  {\left\| {A_{Q^* }^i } \right\|_\infty  }\\
 \le& \| w \|_1 \sum\limits_{i = 0}^\infty  {\gamma ^i }\\
 =& \frac{\| w \|_1}{{1 - \gamma }}
\end{align*}}
and
\[
\left\| v \right\|_\infty   \ge \left\| w \right\|_\infty
\]

\end{proof}

\end{document}